\documentclass[a4paper,11pt]{amsart}
\usepackage{amsmath,amsthm,amssymb,latexsym,enumerate,color,hyperref}
\usepackage{graphicx}
\numberwithin{equation}{section}

\begin{document}

\newtheorem{thm}{Theorem}[section]
\newtheorem{prop}[thm]{Proposition}
\newtheorem{lem}[thm]{Lemma}
\newtheorem{cor}[thm]{Corollary}
\newtheorem{rem}[thm]{Remark}
\newtheorem*{defn}{Definition}

\newcommand{\DD}{\mathbb{D}}
\newcommand{\NN}{\mathbb{N}}
\newcommand{\ZZ}{\mathbb{Z}}
\newcommand{\QQ}{\mathbb{Q}}
\newcommand{\RR}{\mathbb{R}}
\newcommand{\CC}{\mathbb{C}}
\renewcommand{\SS}{\mathbb{S}}

\renewcommand{\theequation}{\arabic{section}.\arabic{equation}}

\def\Xint#1{\mathchoice
{\XXint\displaystyle\textstyle{#1}}%
{\XXint\textstyle\scriptstyle{#1}}%
{\XXint\scriptstyle\scriptscriptstyle{#1}}%
{\XXint\scriptscriptstyle\scriptscriptstyle{#1}}%
\!\int}
\def\XXint#1#2#3{{\setbox0=\hbox{$#1{#2#3}{\int}$}
\vcenter{\hbox{$#2#3$}}\kern-.5\wd0}}
\def\ddashint{\Xint=}
\def\dashint{\Xint-}

\newcommand{\Erfc}{\mathop{\mathrm{Erfc}}}    
\newcommand{\supp}{\mathop{\mathrm{supp}}}    
\newcommand{\re}{\mathop{\mathrm{Re}}}   
\newcommand{\im}{\mathop{\mathrm{Im}}}   
\newcommand{\dist}{\mathop{\mathrm{dist}}}  
\newcommand{\link}{\mathop{\circ\kern-.35em -}}
\newcommand{\spn}{\mathop{\mathrm{span}}}   
\newcommand{\ind}{\mathop{\mathrm{ind}}}   
\newcommand{\rank}{\mathop{\mathrm{rank}}}   
\newcommand{\ol}{\overline}
\newcommand{\pa}{\partial}
\newcommand{\ul}{\underline}
\newcommand{\diam}{\mathrm{diam}}
\newcommand{\lan}{\langle}
\newcommand{\ran}{\rangle}
\newcommand{\tr}{\mathop{\mathrm{tr}}}
\newcommand{\diag}{\mathop{\mathrm{diag}}}
\newcommand{\dv}{\mathop{\mathrm{div}}}
\newcommand{\na}{\nabla}
\newcommand{\nr}{\Vert}

\newcommand{\al}{\alpha}
\newcommand{\be}{\beta}
\newcommand{\ga}{\gamma}  
\newcommand{\Ga}{\Gamma}
\newcommand{\de}{\delta}
\newcommand{\De}{\Delta}
\newcommand{\ve}{\varepsilon}
\newcommand{\fhi}{\varphi} 
\newcommand{\la}{\lambda}
\newcommand{\La}{\Lambda}    
\newcommand{\ka}{\kappa}
\newcommand{\si}{\sigma}
\newcommand{\Si}{\Sigma}
\newcommand{\te}{\theta}
\newcommand{\zi}{\zeta}
\newcommand{\om}{\omega}
\newcommand{\Om}{\Omega}

\newcommand{\cC}{\mathcal{C}}
\newcommand{\cG}{{\mathcal G}}
\newcommand{\cH}{{\mathcal H}}
\newcommand{\cI}{{\mathcal I}}
\newcommand{\cJ}{{\mathcal J}}
\newcommand{\cK}{{\mathcal K}}
\newcommand{\cL}{{\mathcal L}}
\newcommand{\cM}{\mathcal{M}}
\newcommand{\cN}{{\mathcal N}}
\newcommand{\cP}{\mathcal{P}}
\newcommand{\cR}{{\mathcal R}}
\newcommand{\cS}{{\mathcal S}}
\newcommand{\cT}{{\mathcal T}}
\newcommand{\cU}{{\mathcal U}}
\newcommand{\cX}{\mathcal{X}}

\newcommand{\cp}{\sqrt{\frac{p}{p-1}}}
\newcommand{\Cp}{\frac{p}{p-1}}

\title[Asymptotics for the game-theoretic $p$-laplacian]{Asymptotics for the resolvent equation associated to the game-theoretic $p$-laplacian}

\author{Diego Berti}
\address{Dipartimento di Matematica ed Informatica ``U.~Dini'',
Universit\` a di Firenze, viale Morgagni 67/A, 50134 Firenze, Italy.}
    \email{diego.berti@unifi.it}

\author{Rolando Magnanini} 
\address{Dipartimento di Matematica ed Informatica ``U.~Dini'',
Universit\` a di Firenze, viale Morgagni 67/A, 50134 Firenze, Italy.}
    \email{magnanin@math.unifi.it}
    \urladdr{http://web.math.unifi.it/users/magnanin}

\begin{abstract}
We consider the (viscosity) solution $u^\ve$ of the elliptic equation $\ve^2\De_p^G u= u$ in a domain (not necessarily bounded), satisfying $u=1$ on its boundary. Here, $\De_p^G$ is the {\it game-theoretic or normalized $p$-laplacian}.  We derive asymptotic formulas for $\ve\to 0^+$ involving the values of $u^\ve$, in the spirit of Varadhan's work \cite{Va}, and its $q$-mean on balls touching the boundary, thus generalizing that obtained in \cite{MS-AM} for $p=q=2$. As in a related parabolic problem, investigated in \cite{BM}, we link the relevant asymptotic behavior to the geometry of the domain.
\end{abstract}


\maketitle

\raggedbottom

\section{Introduction}
In this paper we consider (viscosity) solutions $u^\ve$ of the following one-parameter family of problems:
\begin{eqnarray}
&u-\ve^2\De_p^G u=0 \ &\mbox{ in } \ \Om, \label{G-elliptic} \\
&u=1  \ &\mbox{ on } \ \Ga, \label{elliptic-boundary}
\end{eqnarray}
where $\Om$ is a domain in $\RR^N$, $N\ge 2$, $\Ga$ is its boundary and $\ve>0$. The operator $\De_p^G$ in \eqref{G-elliptic} is the {\it game-theoretic (or normalized or homogeneous) $p$-laplacian}, that is formally defined by
\begin{equation}
\label{G-laplace}
\De_p^Gu=\frac1{p}\,|\na u|^{2-p}\dv\left\{|\na u|^{p-2}\,\na u\right\}
\end{equation}
if $p\in[1,\infty)$, and by its limit as $p\to\infty$,
\begin{equation}
 \label{infinity-laplacian}
 \De_\infty^G u=\frac{\big\lan \na^2 u\,\na u,\na u\big\ran}{|\na u|^2},
\end{equation}
in the extremal case $p=\infty$. Note that
$\De_2^G=\frac1{2}\De$, where $\De$ is the classical (linear) Laplace operator. Otherwise,  $\De_p^G$ is a non-linear operator in non-divergence form, with possibly discontinuos coefficients at points in which the gradient $\na u$ is zero. Nevertheless, differently from the case of the variational $p$-laplacian $\De_pu=\dv\left\{|\na u|^{p-2}\,\na u\right\}$, $\De_p^G$ inherits from $\De$, for any $p\in(1,\infty]$, its $1$-homogeneity and structure of second-order uniformly elliptic operator. 
The lack of continuity of the coefficients at the critical points of $u$ and the non-divergence form of $\De_p^G$ make the use of the theory of {\it viscosity solutions} prefereable to that of weak (variational) solutions. On this issue, besides the classical treatises \cite{CGG} and \cite{CIL}, we mention a list of more recent references that specifically concern $\De_p^G$: \cite{AP}, \cite{APR}, \cite{BG-IUMJ}, \cite{BG-CPAA}, \cite{Do}, \cite{JK}, \cite{KH}, \cite{KKK}, \cite{MPR} and \cite{MPR2}.
\par
Problems involving $\De_p^G$ appear in a variety of applications. First of all, as the name evokes, the solution has a natural game-theoretic interpretation: the solution of problem \eqref{G-elliptic}-\eqref{elliptic-boundary} represents the limiting value of a two-player, zero-sum game, called  {\it tug-of-war} (for $p=\infty$) 
or {\it tug-of-war with noise} (for $1<p<\infty$), 
see for instance \cite{PS},\cite{PSSW}, \cite{MPR} and \cite{MPR2}.
\par 
In recent years there has been a growing interest for equations involving $\De_p^G$ in relation to numerical methods for {\it image enhancement or restoration} (see \cite{Do} and \cite{BSA}). Typically, for a possibly corrupted image represented by a function $u_0$, it is considered an evolutionary process based on $\De_p^G$ with initial data $u_0$ and homogeneous Neumann boundary conditions. As explained in \cite{Do}, the different choice of $p$ affects in which direction the brightness evolves; the $1$-homogeneity of $\De_p^G$ ensures that such an evolution does not depend on the brightness of the image. The relation between the solution of the parabolic problem and the parametrized elliptic equation \eqref{G-elliptic} is examined in \cite{BSA} (where the backward version of the same problem is considered) for the classical $p$-laplacian, and can be extended to the case of $\De_p^G$ in hand.
\par
More classical applications were considered in the linear case ($p=2$). 
Indeed, in the context of {\it large deviations theory} (see \cite{FW}, \cite{Va}, \cite{EI}), random differential
 equations with small noise intensities are considered and the profile for small values of $\ve$ of the solutions of \eqref{G-elliptic}-\eqref{elliptic-boundary} is related to the behaviour of the exit time of a certain stochastic process. In geometrical terms, that behavior is encoded in the following formula:
\begin{equation}
\label{eq:Va}
\lim_{\ve\to 0^+}
  \ve\log u^\ve(x)=-d_\Ga(x),
\end{equation}
uniformly on a bounded domain $\ol\Om$, where $d_\Ga(x)$ denotes the distance of a point $x\in\Om$ to $\Ga$.
\par
Further evidence of the influence of geometry on the asymptotic behavior of solutions of elliptic and parabolic equations for small values of the relevant parameter was given by the second author of this note and S. Sakaguchi in a series of papers both in the linear case (\cite{MS-AM}, \cite{MS-IUMJ}, \cite{MS-JDE}, \cite{MS-MMAS}, \cite{MM-NA}) and in certain non-linear contexts (\cite{MS-PRSE}, \cite{MS-AIHP}, \cite{MS-JDE2}, \cite{Sa}), concerning both initial-boundary value problems (\cite{MS-AM}, \cite{MS-PRSE}, \cite{MS-AIHP}) and initial-value problems (\cite{MPS}, \cite{MPrS}), and even for two-phase problems (\cite{Sa2}, \cite{Sa3}, \cite{CaMS}). For instance, in \cite[Theorem 2.3]{MS-AM}, in the case $p=2$ for problem \eqref{G-elliptic}-\eqref{elliptic-boundary}, the following formula involving the mean value of $u^\ve$ was established: 
\begin{equation}
 \label{eq:MS}
\lim_{\ve\to 0^+}
\left(\frac{R}{\ve}\right)^{\frac{N-1}{2}}\dashint_{\pa B_R(x)}u^\ve(y)\,dS_y=\frac{c_N}{\sqrt{\Pi_\Ga(y_x)}};
\end{equation}
$c_N$ is a constant that can be derived from \cite[Theorem 2.3]{MS-AM}. Here, for a given $x\in\Om$ with $d_\Ga(x)=R$, $B_R(x)$ is assumed to be a ball contained in $\Om$ and such that $(\RR^N\setminus\Om)\cap\pa B_R(x)=\{ y_x\}$, $\ka_1,\dots,\ka_{N-1}$ denote the principal curvatures (with respect to the inward normal) of $\Ga$ at points of $\Ga$, and  
\begin{equation}
 \label{eq:Pi-gamma}
\Pi_\Ga(y_x)=\prod\limits_{j=1}^{N-1}
\Bigl[1-R\,\ka_j(y_x)\Bigr].
\end{equation}
\par 
In this paper, we shall derive asymptotic formulas similar to \eqref{eq:Va} and \eqref{eq:MS} for the solution $u^\ve$ of \eqref{G-elliptic}-\eqref{elliptic-boundary}, when $p\not=2$.
\par
First of all, in Section \ref{sec:first-order} we shall prove for $p\in(1,\infty]$ the validity of the following analog of \eqref{eq:Va}: 
\begin{equation}
\label{eq:elliptic-first-order}
 \lim_{\ve\to 0^+}
\ve\log \left[u^\ve(x)\right]=-\sqrt{p'}\,d_\Ga(x), \ \mbox{ for any } \ x\in\ol{\Om}.
\end{equation}
Here $p'$ denotes the conjugate exponent to $p$ (i.e. $1/p+1/p'=1$) and $\Om$ is required to merely satisfy the regularity assumption: $\Ga=\pa\left(\RR^N\setminus\ol\Om\right)$.
\par 
Asymptotics \eqref{eq:elliptic-first-order} is obtained, first, by considering the solutions of \eqref{G-elliptic}-\eqref{elliptic-boundary} in the two basic radial cases in which $\Om$ is either a ball or the exterior of its closure  (see Subsection \ref{sec:radial-case}) and, secondly, by employing them as barriers, by virtue of the comparison principle and taking advantage of the fact that in the radial cases the solutions of \eqref{G-elliptic}-\eqref{elliptic-boundary} are obtained explicitly. In fact, equation \eqref{G-elliptic} can be reduced to a {\it linear} ordinary differential equation when $u^\ve$ is radially symmetric.  The use of radial barriers, that is natural in the case of a $C^2$ domain, will also be adapted to a quite general case.
\par 
In Section \ref{sec:first-order} we shall also
investigate uniform estimates related to \eqref{eq:elliptic-first-order}. 
An interesting feature of such estimates is that their quality seems to depend on the parameter $p$. 
In fact, while for a quite large class of regular domains (that we call $C^{0,\om}$ domains and these include the scale of those with H\"older regularity) the following formula always holds uniformly on any compact subset of $\ol\Om$
\begin{equation}
\label{eq:uniform}
\ve\log\left(u^\ve \right)+\sqrt{p'}\,d_\Ga=O\left(\ve|\log \psi(\ve)|\right) \ \mbox{ as } \ \ve\to 0^+
\end{equation}
(here $\psi(\ve)=\psi_\om(\ve)$ vanishes with $\ve$ depending on the modulus of continuity $\om$), we see that, the right-hand side of \eqref{eq:uniform} improves to $O(\ve\log|\log\psi(\ve)|)$ if $p=N$, $O(\ve \log\ve)$ for $p\in (N,\infty)$,  and $O(\ve)$ if $p=\infty$ (see Theorem \ref{th:uniform-elliptic}).
\par 
In Section \ref{sec:asymptotics-second-order}, we shall derive a formula that greatly generalizes \eqref{eq:MS}: in fact, in Theorem \ref{th:q-mean} it is shown that for $1<p\le\infty$ and $1<q<\infty$ it
holds that
\begin{equation}
\label{q-mean}
\lim_{\ve\to 0^+}
\left(\frac{R}{\ve}\right)^{\frac{N+1}{2(q-1)}}\!\!\mu_{q,\ve}(x)=\frac{c_{N,q}}{\left\{(p')^{\frac{N+1}{2}}
\Pi_\Ga(y_x)\right\}^{\frac{1}{2(q-1)}}}
\end{equation}
(the exact value of the constant $c_{N,q}$ is given in Theorem \ref{th:q-mean}). Here, $\mu_{q,\ve}(x)$ is the {\it $q$-mean of} $u^\ve$ on a ball $B_R(x)$ touching $\Ga$ at only one point $y_x$; $\mu_{q,\ve}(x)$ is the unique $\mu\in\RR$ such that 
\begin{equation}
\label{q-mean-definition} 
\nr u^\ve-\mu\nr_{L^q(B_R(x))}\le\nr u^\ve-\la\nr_{L^q(B_R(x))} \ \mbox{ for all } \ \la\in\RR;
\end{equation}
this can be defined for every $q\in[1,\infty]$ (see \cite{IMW}). Notice that $\mu_{2,\ve}(x)$ is the standard mean value of $u^\ve$ on $B_R(x)$.
\par
 If $q=\infty$ we simply have that $\mu_{\infty,\ve}(x)\to 1/2$, which does not give any geometrical information. 
\par 
 Also to obtain \eqref{q-mean}, we first show that an appropriate formula holds for suitable spherical barriers for $u^\ve$ and hence we use the fact that $q$-means are monotonic with respect to (almost everywhere) partial order of functions on $B_R(x)$.
\par
Formulas \eqref{eq:elliptic-first-order} and \eqref{q-mean} are elliptic versions of similar formulas obtained in \cite{BM} for the viscosity solution $u=u(x,t)$ of the parabolic equation
\begin{equation*}
\label{G-parabolic}
u_t-\De_p^Gu=0 \ \mbox{ in } \ \Om\times(0,\infty),
\end{equation*}
subject to the boundary-initial data:
\begin{equation*}
\label{parabolic-boundary}
u=0\ \mbox{ on }\ \Om\times\{0\} \ \mbox{ and }\  u=1 \ \mbox{ on } \ \Ga\times(0,\infty).
\end{equation*}
Indeed, in \cite{BM} we obtained:
\begin{equation*}
  \label{eq:short-time-first-asymptotics}
  \lim_{t\to 0^+}4t\,\log u(x,t)=-p'\,d_\Ga(x)^2,
\end{equation*}
and 
\begin{equation*}
  \label{eq:parabolic-q-mean}
\lim_{t\to 0^+}
\left(\frac{R^2}{t}\right)^{\frac{N+1}{4(q-1)}}\mu_{q,t}(x)=\frac{C_{N,q}}{\left\{(p')^{\frac{N+1}{2}}\Pi_\Ga(y_x)\right\}^{\frac1{2(q-1)}}},
\end{equation*}
for some explicit constant $C_{N,q}$, where $\mu_{q,t}(x)$ is the $q$-mean of $u(\cdot,t)$ on $B_R(x)$.  It should be noticed that in the case $p=2$, due to the linearity of $\De$, the last formula and \eqref{q-mean} can be obtained from one another, since $u^\ve(x)$ and $u(x,t)$ are related by a Laplace transformation. When $p\not=2$, this is no longer possible and the elliptic and parabolic cases must be treated separately. Moreover, in the elliptic case, due to the availability of explicit barriers, 
we obtain the more accurate uniform estimate \eqref{eq:uniform} (compare with \cite[Eq. (2.18)]{BM}).
\par
We conclude this introduction by mentioning that the linearity of $\De$  was used in \cite{MS-AM} to derive radial symmetry of the so-called {\it stationary isothermic surfaces}, that is those level surfaces of the temperature which are invariant in time. In fact, it was shown that the mean value $\mu_{q,t}(x)$ or   $\mu_{q,\ve}(x)$ does not depend on $x$ if this lies on a stationary isothermic surface, and hence, for instance, \eqref{eq:MS} gives that $\Pi_\Ga$ must be constant on $\Ga$, so that radial symmetry then ensues from Alexandrov's Soap Bubble Theorem for Weingarten surfaces (see \cite{Al}). 
For $p\not=2$, this approach is no longer possible. However, an approach based on the method of moving planes, as considered in \cite{MS-AIHP} and \cite{CMS-JEMS}, may still be possible.

\section{Large-deviations asymptotics}
\label{sec:first-order}

The proof of the asymptotic formulas \eqref{eq:elliptic-first-order} and \eqref{eq:uniform} will be carried out in two steps.

\subsection{Asymptotics in the radial cases}
\label{sec:radial-case}

In this subsection we collect the relevant properties of the explicit solutions for problem \eqref{G-elliptic}-\eqref{elliptic-boundary} when $\Om$ is a ball or the complement of a ball. We shall denote by $B_R$ the ball with radius $R$ centered at the origin. 
\par 
In \cite{BM} we have shown that the (viscosity) solution of \eqref{G-elliptic}-\eqref{elliptic-boundary} in $B_R$ is given by
\begin{equation} 
\label{ball solution formula}
u^\ve(x)=\begin{cases}
\displaystyle 
\frac{\int_{0}^{\pi}e^{\sqrt{p'}\,\frac{|x|}{\ve}\,\cos \te}(\sin\te) ^{\al}\,d\te}{\int_{0}^{\pi}e^{\sqrt{p'}\,\frac{R}{\ve}\,\cos\te}(\sin\te) ^{\al}\,d\te} \ &\mbox{ if } \ 1<p<\infty,
\vspace{7pt} \\
\displaystyle
\frac{\cosh(|x|/\ve)}{\cosh(R/\ve)}  \ &\mbox{ if } \ p=\infty,
\end{cases} 
\end{equation}
for $x\in\ol{B}_R$; here 
\begin{equation}
\label{def-alpha}
\al=\frac{N-p}{p-1}=-1+\frac{N-1}{p-1}
\end{equation}
 Moreover, formula \eqref{eq:elliptic-first-order} holds with $\Om=B_R$ and the next lemma gives a quantitative estimate of the convergence in \eqref{eq:elliptic-first-order}. 

\begin{lem}[Uniform asymptotics in a ball]
\label{lem:uniform-estimate-ball}
Let $1<p\le\infty$ and let  $u^\ve$ be the solution of \eqref{G-elliptic}-\eqref{elliptic-boundary} in $B_R$.
\par
Then
\begin{equation}
\label{uniform-ball}
\ve\log u^\ve+d_\Ga=
\begin{cases}
\displaystyle 
O(\ve\,|\log\ve|) \ &\mbox{ if } \ 1<p<\infty,
\\
\displaystyle
O(\ve)  \ &\mbox{ if } \ p=\infty,
\end{cases} 
\end{equation}
uniformly on $\ol{B_R}$ as $\ve\to 0^+$.
\end{lem}

\begin{proof}
The case $p=\infty$ follows at once, since \eqref{ball solution formula} gives:
$$
\ve\log\left\{u^\ve(x)\right\}+d_\Ga(x)=\ve\log\left[\frac{1+e^{-2\frac{|x|}{\ve}}}{1+e^{-2\frac{R}{\ve}}}\right].
$$
\par 
If $1<p<\infty$, by \eqref{ball solution formula} we have that
\begin{equation*}
\ve\log\left\{u^\ve(x)\right\}+\sqrt{p'}\,d_\Ga(x)=
\ve\,\log\left[\frac{\int_{0 }^{\pi}e^{-\sqrt{p'}(1-\cos\te)\,\frac{|x|}{\ve}}(\sin\te)^\al\,d\te}{\int_{0}^{\pi}e^{-\sqrt{p'}(1-\cos\te)\,\frac{R}{\ve}}(\sin\te)^\al\,d\te}\right]
\end{equation*}
and the right-hand side is decreasing in $|x|$, so that 
\begin{equation*}
0
\le
\ve\,\log\left\{u^\ve(x)\right\}+\sqrt{p'}\,d_\Ga(x)\le 
\ve\,\log\left[
\frac{\int_{0}^{\pi}(\sin\te)^\al\,d\te}
{\int_{0}^{\pi}e^{-\sqrt{p'}\,(1-\cos\te)\,\frac{R}{\ve}}(\sin\te)^\al\,d\te}
\right].
\end{equation*}
This formula gives \eqref{uniform-ball}, since we have that 
\begin{multline*}
\int_{0}^{\pi}e^{-\sqrt{p'}\,(1-\cos\te)\frac{R}{\ve}}(\sin\te)^\al\,d\te=\\
2^{\frac{\al-1}{2}}\Ga\left(\frac{\al+1}{2}\right)\left(\frac{R\sqrt{p'}}{\ve}\right)^{-\frac{\al+1}{2}}\,[1+O(\ve)]
\end{multline*}
as $\ve\to 0^+$, by using Lemma \ref{lem:technical}.
\end{proof}

\medskip

Next, we consider the complement of a ball.

\begin{lem}[Uniform asymptotics in the complement of a ball]
\label{elliptic-from-below}
Set $1<p\le \infty$, $\Om=\RR^N\setminus\ol{B_R}$, and let $u^\ve$ be
the bounded solution of \eqref{G-elliptic}-\eqref{elliptic-boundary}.
\par
Then we have that
\begin{equation}
\label{solution-exterior}
u^\ve(x)=
\begin{cases}
\displaystyle
\frac{\int_{0}^{\infty}e^{-\sqrt{p'}\cosh\te\frac{|x|}{\ve}}(\sinh\te)^\al\,d\te}{\int_{0}^{\infty}e^{-\sqrt{p'}\cosh\te \frac{R}{\ve}}(\sinh\te)^\al\,d\te}\ &\mbox{ if }\ 1<p<\infty,  
\vspace{7pt}\\
\displaystyle
e^{-\frac{|x|-R}{\ve}} \ &\mbox{ if } \ p=\infty.
\end{cases}
\end{equation}
\par 
In particular,
\begin{equation}
\label{eq:uniform-exterior-estimate}
  \ve\log\left\{u^\ve(x)\right\}+\sqrt{p'}\,d_\Ga(x)=O(\ve)  \ \mbox{ as } \ \ve\to 0^+,
\end{equation}
uniformly on every compact subset of $\ol\Om$.
\end{lem}

\begin{proof}
It is immediate to verify that \eqref{elliptic-boundary} is satisfied. Moreover, it is a straightforward computation to show that $u^\ve$ satisfies \eqref{G-elliptic} in the classical sense, once we observe that it is of class $C^2$ for $|x|>R$.
This is enough to conclude that $u^\ve$ is the unique viscosity solution of \eqref{G-elliptic}-\eqref{elliptic-boundary}.
\par
First, notice that $d_\Ga(x)=|x|-R$ for $|x|\ge R$. If $p=\infty$, \eqref{eq:uniform-exterior-estimate} holds exactly as
$$
\ve\log\left\{u^\ve(x)\right\}+d_\Ga(x)\equiv 0.
$$
\par
If $1<p<\infty$, we write that
$$
\ve\log\left\{u^\ve(x)\right\}+\sqrt{p'}\,d_\Ga(x)=
\frac{\int_{0}^{\infty}e^{-\sqrt{p'}\frac{|x|}{\ve}(\cosh\te-1)}(\sinh\te)^\al\,d\te}{\int_{0}^{\infty}e^{-\sqrt{p'}\frac{R}{\ve}(\cosh\te-1)}(\sinh\te)^\al\,d\te}
$$
and hence, by monotonicity,  we have that
$$
\ve\,\log\left\{\frac{\int_{0}^{\infty}e^{-\sqrt{p'}\frac{R'}{\ve}(\cosh\te-1)}(\sinh\te)^\al\,d\te}{\int_{0}^{\infty}e^{-\sqrt{p'}\frac{R}{\ve}(\cosh\te-1)}(\sinh\te)^\al\,d\te}\right\}\le\ve\log\left\{u^\ve(x)\right\}+\sqrt{p'}\,d_\Ga(x)\le 0,
$$
for every $x$ such that $R\le|x|\le R'$, with $R'>R$.
Our claim then follows by an inspection on the left-hand side, after applications for $\si=\sqrt{p'} R'$ and $\si=\sqrt{p'} R$ of Lemma \ref{lem:technical}.
\end{proof}

\subsection{Asymptotics in a general domain}
\label{sec:first-order-general-domain}

We begin by recalling that for equation \eqref{G-elliptic} the {\it comparison principle} holds, as noted in \cite[Remark 4.6]{CGG} or in \cite[Appendix D]{AP} and as shown in \cite[Theorem 2.1]{Sat} in the general case of bounded solutions on unbounded domains, that is if
$u$ and $v$ are viscosity solutions of \eqref{G-elliptic} in $\Om$ such that $u\le v$ on $\Ga$, then
$u\le v$ on $\ol\Om$.

The next two lemmas give explicit barriers for the solution in a general domain $\Om$. We observe that no regularity assumption on $\Om$ is needed. 

\begin{lem}[Control from above]
  \label{elliptic-control-from-above}
Let $1<p\le\infty$ and $u^\ve$ be the bounded (viscosity) solution of \eqref{G-elliptic}-\eqref{elliptic-boundary}. 
\par 
Then, we have that
$$
\ve\log\left\{u^\ve(x)\right\}+\sqrt{p'}\,d_\Ga(x)\leq \ve\log E_p^\ve(d_\Ga(x)),
$$
for every $x\in\ol\Om$, where
\begin{equation}
\label{error-from-above}
E_p^\ve(\si)=
\begin{cases}
\displaystyle
\frac{\int_{0}^{\pi}(\sin\te)^\al\,d\te}{\int_{0}^{\pi}e^{-\sqrt{p'}\frac{1-\cos\te}{\ve} \si}(\sin\te)^\al\,d\te}\ &\mbox{ if }\ 1<p<\infty,
\vspace{8pt}\\
\displaystyle
\frac{2}{1+e^{-\frac{2 \si}{\ve}}}\ &\mbox{ if }\ p=\infty.
\end{cases}
\end{equation}
\par
In particular, it holds that
\begin{equation*}
\ve \log E_p^\ve(d_\Ga)=
\begin{cases}
O(\ve\,\log\ve) \ \mbox{ if } 1<p<\infty,\\
O(\ve)\ \mbox{ if } p=\infty,
\end{cases}
\end{equation*}
as $\ve\to 0^+$, on every subset of $\Om$ in which $d_\Ga$ is bounded.
\end{lem}

\begin{proof}
For a fixed $x\in\Om$, we consider the ball $B^x=B_R(x)$ with $R=d_\Ga(x)$ and denote by $u_{B^x}^\ve$ the solution of \eqref{G-elliptic}-\eqref{elliptic-boundary} with $\Om=B^x$. The comparison principle gives that
$$
u^\ve\leq u_{B^x}^\ve\ \mbox{ on }\ \ol{B^x}
$$  
and, in particular, 
\begin{equation}
\label{eq:ball-inside}
u^\ve(x)\leq u_{B^x}^\ve(x).
\end{equation}
Observe that the uniqueness of  the solution of \eqref{G-elliptic}-\eqref{elliptic-boundary} and the scaling properties of $\De_p^G$ imply that
$$
u_{B^x}^\ve(x)=u_B^{\ve/R}(0),
$$
where $u_B^\ve$ is the solution of \eqref{G-elliptic}-\eqref{elliptic-boundary} with $\Om=B$, the unit ball. The explicit expression in \eqref{ball solution formula} and \eqref{eq:ball-inside} then yield the desired claim, since $R=d_\Ga(x)$.
\par
The last claim follows from Lemma \ref{lem:uniform-estimate-ball}.
\end{proof}

\begin{lem}[Control from below]
\label{elliptic-control-from-below}
Let $1<p\le\infty$ and $u^\ve$ be the bounded (viscosity) solution of \eqref{G-elliptic}-\eqref{elliptic-boundary}.  Pick $z\in\RR^N\setminus\ol\Om$.
\par
Then, we have that
$$
\ve\log\left\{u^\ve(x)\right\}+\sqrt{p'}\,\{|x-z|-d_\Ga(z)\}\geq \ve\log e_{p,z}^\ve(x)
\ \mbox{ for any } \ x\in\ol\Om,
$$
where
\begin{equation}
\label{error-from-below}
e_{p,z}^\ve(x)=
\begin{cases}
\displaystyle
\frac{\int_{0}^{\infty}e^{-\sqrt{p'}\frac{\cosh\te-1}{\ve}|x-z|}\left(\sinh\te\right)^{\al}d\te}
{\int_{0}^{\infty}e^{-\sqrt{p'}\frac{\cosh\te-1}{\ve}d_\Ga(z)}\left(\sinh\te\right)^{\al}d\te} \ &\mbox{ if }\ 1<p<\infty,
\vspace{7pt}\\
\displaystyle
1  \ &\mbox{ if } \ p=\infty.
\end{cases}
\end{equation}
\end{lem}

\begin{proof}
 We consider the ball $B=B_R(z)$ with radius $R=d_\Ga(z)$ and let $v^\ve$ be the bounded solution of \eqref{G-elliptic}-\eqref{elliptic-boundary} relative to $\RR^N\setminus\ol{B}\supset \Om$. From the fact that $z\in\RR^N\setminus\ol\Om$, we have that $\Ga\subset \RR^N\setminus B$, which implies that
$$
v^\ve\leq 1 \ \mbox{ on }\ \Ga,
$$
by the explicit expression of $v^\ve$ given in \eqref{solution-exterior}. Thus, by the comparison principle, we infer that $v^\ve\leq u^\ve$ on $\ol\Om$. The desired claim then follows by easy manipulations on \eqref{solution-exterior}.
\end{proof}

\begin{thm}[Pointwise convergence]
\label{th:pointwise}
Let $1<p\le \infty$ and $\Om$ be a domain satisfying $\Ga=\pa\left(\RR^N\setminus\ol{\Om}\right)$; assume that $u^\ve$ is the bounded (viscosity) solution of \eqref{G-elliptic}-\eqref{elliptic-boundary}.
\par 
Then, it holds that
\begin{equation}
  \label{eq:pointwise-elliptic}
  \lim_{\ve\to 0^+}
\ve\log\left\{u_p^\ve(x)\right\}=
-\sqrt{p'}\,d_\Ga(x) \ \mbox{ for any } \ x\in\ol\Om.
\end{equation}
\end{thm}

\begin{proof}
 Given $z\in\RR^N\setminus\ol\Om$ and $\ve>0$, combining Lemmas \ref{elliptic-control-from-above} and  \ref{elliptic-control-from-below} gives at $x\in\ol{\Om}$ that
\begin{multline}
\label{barrier-chain}
\sqrt{p'}\left\{-|x-z|+d_\Ga(x)+d_\Ga(z)\right\}+\ve\log e_{p,z}^\ve(x)\leq\\
\ve\log\{u^\ve(x)\} +\sqrt{p'}\,d_\Ga(x)\leq
\ve \log E_p^\ve(d_\Ga(x)).
\end{multline}
Letting $\ve\to 0^+$ then gives that
\begin{multline*}
\sqrt{p'}\left\{ -|x-z|+d_\Ga(x)+d_\Ga(z)\right\}\le\\
 \liminf_{\ve\to 0^+}\left[ \ve\log\{u^\ve(x)\}+\sqrt{p'}\,d_\Ga(x)\right] \leq\\
\limsup_{\ve\to 0^+}\left[\ve\log\{u^\ve(x)\}+\sqrt{p'}\,d_\Ga(x)\right]\le
0, 
\end{multline*}
where we have used Lemma  \ref{elliptic-control-from-above} and the fact that $\ve\log e_{p,z}^\ve(x)$ vanishes, as $\ve\to 0^+$, by applying Lemma \ref{lem:technical} to \eqref{error-from-below}.
\par 
Now, since $z$ is arbitrary in $\RR^N\setminus\ol\Om$, we obtain \eqref{eq:pointwise-elliptic}, by taking the limit for $z\to y$,
where $y\in\Ga$ is a point realizing $|x-y|=d_\Ga(x)$.
\end{proof}

\subsection{Uniform asymptotics}
For a domain of class $C^0$, we mean that its boundary is locally the graph of a continuous function. For the sequel, it is convenient to specify the modulus of continuity, by the following definition (introduced in \cite{BM}). Let $\om:(0,\infty)\to (0,\infty)$ be a strictly increasing continuous function such that $\om(\tau)\to 0$ as $\tau\to 0^+$. We say that a domain $\Om$ is of class $C^{0,\om}$, if there exists a number $r>0$ such that, for every point $x_0\in\Ga$, there is a coordinate system $(y',y_N)\in\RR^{N-1}\times\RR$, and a function $\phi:\RR^{N-1}\to\RR$ such that
\begin{enumerate}[(i)]
\item
$B_r(x_0)\cap\Om=\{(y',y_N)\in B_r(x_0):y_N<\phi(y')\}$;
\item
$B_r(x_0)\cap\Ga=\{(y',y_N)\in B_r(x_0):y_N=\phi(y')\}$
\item
$|\phi(y')-\phi(z')|\le\om(|y'-z'|)$ for all $(y',\phi(y')), (z',\phi(z'))\in B_r(x_0)\cap\Ga$.
\end{enumerate}
In the sequel, it will be useful the function defined for $\ve>0$ by
$$
\psi(\ve)=\min_{0\leq s \leq r}\sqrt{s^2+\left[\om(s)-\ve\right]^2}
$$
--- this is the distance of the point $z_\ve=(0',\ve)\in\RR^{N-1}\times\RR$ from the graph of the function $\om$.
Notice that $\psi(\ve)=\ve$ if $\phi\in C^k$ with $k\ge 2$ and, otherwise, $\psi(\ve)\sim C\,\om^{-1}(\ve)$, for some positive constant $C$, where $\om^{-1}$ is the inverse function of $\om$.

\begin{thm}[Uniform convergence]
\label{th:uniform-elliptic}
Let $1<p\le \infty$ and $\Om$ be a domain of class $C^0$. Suppose that $u^\ve$ is the bounded (viscosity) solution of \eqref{G-elliptic}-\eqref{elliptic-boundary}.
\par
Then, as $\ve\to 0^+$, we have that 
\begin{equation}
\label{eq:uniform-convergence-infinity}
\ve\log\left\{u^\ve(x)\right\}+d_\Ga(x)=
\begin{cases}
O(\ve) \ &\mbox{ if } \ p=\infty, \\
O(\ve\log \ve) \ &\mbox{ if } \ p>N.
\end{cases}
\end{equation}
Moreover,  if $\Om$ is a $C^{0,\om}$ domain, it holds that
\begin{equation}
  \label{eq:uniform-convergence}
  \ve\,\log \left\{u^\ve(x)\right\}+\sqrt{p'}\,d_\Ga(x)=
  \begin{cases}
  \displaystyle
    O(\ve\log|\log\psi(\ve)|)\ &\mbox{ if } \ N=p, \\  
\displaystyle 
    O(\ve\log\psi(\ve))\ &\mbox{ if }\ 1<p<N.
  \end{cases}
\end{equation}
The formulas \eqref{eq:uniform-convergence-infinity} and \eqref{eq:uniform-convergence} hold uniformly on the compact subsets of $\ol\Om$.
\par
In particular, if $\ve\log\psi(\ve)\to 0$ as $\ve\to 0^+$, then the convergence in \eqref{eq:pointwise-elliptic} is uniform on every compact subset of $\ol\Om$.
\end{thm}

\begin{proof}
For any fixed compact subset $K$ of $\ol\Om$ we let $d$ be the positive number, defined as  
$$
d=
\max_{x'\in K}\{d_\Ga(x'), |x'|\}.
$$ 
To obtain the uniform convergence in \eqref{eq:pointwise-elliptic} we will choose $z=z_\ve$ independently on $x\in K$, as follows.
\par 
If $\Om$ is a $C^{0,\om}$ domain, fix $x\in K$, take $y\in\Ga$ minimizing the distance to $x$, and 
consider a coordinate system in $\RR^{N-1}\times\RR$ such that $y=(0', 0)$. If we take $z_\ve=(0',\ve)$, then $z_\ve\in\RR^N\setminus\ol\Om$ when $\ve$ is sufficiently small. 
With this choice, \eqref{barrier-chain} reads as
\begin{multline*}
\sqrt{p'}\left\{-|x-z_\ve|+d_\Ga(x)+d_\Ga(z_\ve)\right\}+\ve\log e_{p,z_\ve}^\ve(x)\leq\\
\ve\log u_p^\ve(x) +\sqrt{p'}\,d_\Ga(x)\leq
\ve \log E_p^\ve(d_\Ga(x)).
\end{multline*}
Hence, we get:
\begin{equation}
\label{barrier-chain-uniform}
-\sqrt{p'}\,\ve+\ve\log e_{p,z_\ve}^\ve(x)\leq
\ve\log u_p^\ve(x) +\sqrt{p'}\,d_\Ga(x)\leq
\ve \log E_p^\ve(d_\Ga(x)),
\end{equation}
since $d_\Ga(z_\ve)\ge 0$ and $|x-z_\ve|\le d_\Ga(x)+\ve$. 
\begin{figure}
  \centering
  \includegraphics[scale=0.6]{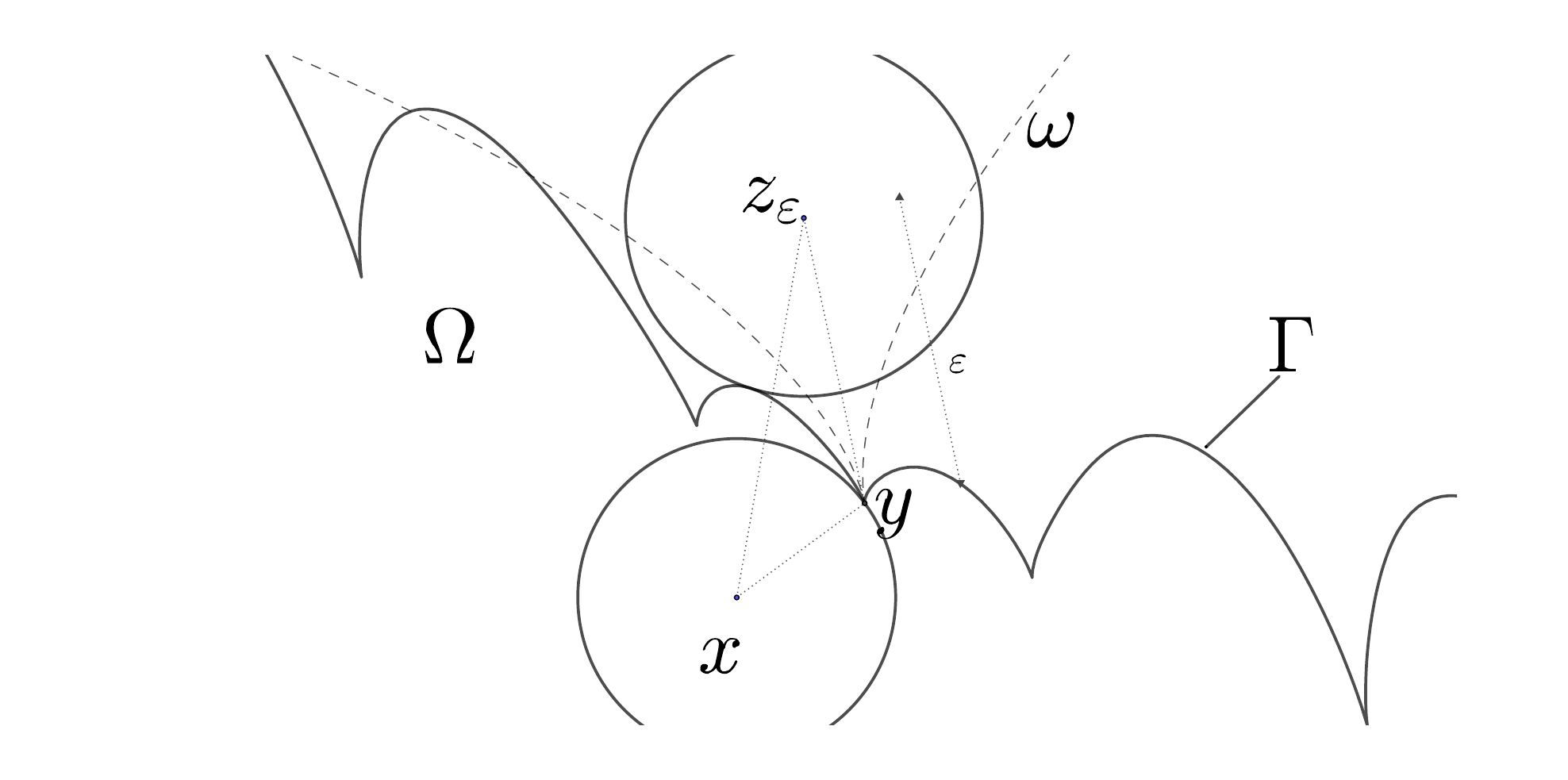} 
  \caption{The geometric description of the approximating scheme in the proof of Theorem \ref{th:uniform-elliptic}.}
  \label{fig:picture}
\end{figure}
\par
Thus, if $p=\infty$, Lemmas \ref{elliptic-control-from-above} and \eqref{elliptic-control-from-below} give that
$$
-\ve\le \ve \log \left\{u^\ve(x)\right\}+d_\Ga(x)\le \ve\log\left\{\frac{2}{1+e^{-\frac{d}{\ve}}}\right\},
$$
being $d_\Ga(x)\le d$, and \eqref{eq:uniform-convergence-infinity} follows at once.

Next, 
if $1<p<\infty$, we recall that $\ve\log E_p^\ve(d_\Ga)=O(\ve\log \ve)$ on $K$ as $\ve\to 0^+$, by Lemma \ref{elliptic-control-from-above}. On the other hand, by observing that $d_\Ga(z_\ve)\ge \psi(\ve)$, by our assumption on $\Om$, and that also $|x-z_\ve|\le 2 d$ for $\ve\le d$,
\eqref{error-from-below} gives on $K$ that
$$
  e_{p,z_\ve}^\ve\ge
 \frac{\int_{0}^{\infty}e^{-\frac{2 d\,\sqrt{p'}}{\ve}(\cosh\te-1)}\left(\sinh\te\right)^{\al}d\te}
{\int_{0}^{\infty}e^{-\frac{\sqrt{p'}\psi(\ve)}{\ve}(\cosh\te-1)}\left(\sinh\te\right)^{\al}d\te}.
$$
\par
Now, to this formula we apply Lemma \ref{lem:technical} with $\si=2 d \sqrt{p'}/\ve$ at the numerator and $\si=\sqrt{p'} \psi(\ve)/\ve$ at the denominator. Thus, since by \eqref{def-alpha} the sign of $\al$ is that of $N-p$, on $K$ we have as $\ve\to 0$  that
\begin{multline*}
 \ve\log\left(e_{p,z_\ve}^\ve\right)\ge \al\,\ve\log\psi(\ve)-\frac{\al-1}{2}\,\ve\log\ve+O(\ve)=\\
\al\ve\log\psi(\ve)+O(\ve\log\ve),
\end{multline*}
if $p<N$,
$$
 \ve\log\left(e_{p,z_\ve}^\ve\right)\ge -\ve\,\log|\log\psi(\ve)|+O(\ve \log\ve),
$$
if $p=N$, and
$$
 \ve\log\left(e_{p,z_\ve}^\ve\right)\ge \frac{\al+1}{2}\,\ve\log\ve+O(\ve),
$$
if $p>N$.
\end{proof}

\section{Asymptotics for the $q$-means of a ball}
\label{sec:asymptotics-second-order}
Throughout this subsection we assume that $\Om$ is a (not necessarily bounded) domain that satisfies  both the uniform exterior and interior ball conditions, i.e. there exist $r_i,r_e>0$ such that every $y\in\Ga$ has the property that there exist $z_i\in\Om$ and $z_e\in\RR^{N}\setminus\ol{\Om}$ for which 
\begin{equation}
\label{uniform-balls-condition}
B_{r_i}(z_i)\subset \Om \subset \RR^{N}\setminus \ol{B}_{r_e}(z_e)\ \mbox{ and } \ \ol{B}_{r_i}(z_i)\cap\ol{B}_{r_e}(z_e)=\{y\}.
\end{equation}

\subsection{Enhanced barriers}

We begin by refining the barriers given in Lemma \ref{elliptic-control-from-above} and \ref{elliptic-control-from-below}, at least in a strip near the boundary $\Ga$. To this aim, we will use the notation:
$$
\Om_\rho=\{y\in\Om:\, d_\Ga(y)\leq \rho\}, \ \rho>0.
$$
We will also use the two families of probability measures on the intervals $[0, \infty)$ and $[0,\pi]$ with densities defined, respectively, by
\begin{eqnarray*}
&&\displaystyle d\nu^\si(\te)=\frac{e^{-\si\,(\cosh\te-1)}(\sinh\te)^\al}{\int_{0}^{\infty}e^{-\si (\cosh\te-1)}(\sinh\te)^\al\,d\te}\,d\te,\\
&&\displaystyle d\mu^\si(\te)=\frac{e^{-\si\,(1-\cos\te)}(\sin\te)^\al}{\int_{0}^{\pi}e^{-\si (1-\cos\te)}(\sin\te)^\al\,d\te}\,d\te.
\end{eqnarray*}

\begin{lem}
  \label{barriers-for-c2-domain}
Let $u^\ve$ be the bounded (viscosity) solution of \eqref{G-elliptic}-\eqref{elliptic-boundary}.
\par 
If $p\in (1,\infty)$, we set for $\si_\ve=\sqrt{p'} r_e/\ve$:
$$
U^\ve(\tau)=
\int_{0}^{\infty}e^{-\tau\cosh\te} d\nu^{\si_\ve}(\te), \quad \tau\ge 0,
$$
and
\begin{equation*}
V^\ve(\tau)=
\begin{cases}
\displaystyle
\int_{0}^{\pi}e^{-\tau\cos\te} d\mu^{\si_\ve}(\te)\ &\mbox { if }\ 0\le \tau < \si_\ve, \\
\vspace{-10pt}\\ 
\displaystyle
\left\{\int_{0}^{\pi}e^{-\tau\cos\te} d\mu^0(\te)\right\}^{-1}\ &\mbox{ if }\ \tau\ge  \si_\ve.
\end{cases}
\end{equation*}
If $p=\infty$, we set $U^\ve(\tau)=e^{-\tau}$ and 
\begin{equation*}
V^\ve(\tau)=
\begin{cases}
\displaystyle
\frac{\cosh(\si_\ve-\tau)}{\cosh\si_\ve}\ &\mbox{ if } \ 0\le \tau < \si_\ve, \\
\vspace{-8pt}\\
1/\cosh\tau\ &\mbox{ if }\ \tau\ge \si_\ve.
\end{cases}
\end{equation*}
\par
Then, we have that
\begin{equation}
\label{eq:U-V-barriers}
U^\ve\left(\frac{d_\Ga(x)}{\ve/\sqrt{p'}}\right)\le 
u^\ve(x)\le
V^\ve\left(\frac{d_\Ga(x)}{\ve/\sqrt{p'}}\right),
\end{equation}
for any $x\in\ol\Om$.
\end{lem}

\begin{proof}
  Let $p\in(1,\infty)$. For any $x\in\Om$ we can consider $y\in \Ga$ such that $|x-y|=d_\Ga(x)$. From the assumptions on $\Om$ there exists $z_e\in\RR^N\setminus\ol\Om$ such that \eqref{uniform-balls-condition} holds for $y$. As seen in the proof of Lemma \ref{elliptic-control-from-below}, by using the comparison principle and the explicit expression \eqref{solution-exterior}, we obtain
\begin{equation*}
u^\ve(x)\ge\frac{\int_0^\infty e^{-\sqrt{p'} |x-z_e|/\ve\,\cosh\te}(\sinh\te)^\al\,d\te}{\int_{0}^{\infty}e^{-\sqrt{p'} r_e/\ve\,\cosh\te}(\sinh\te)^\al\,d\te}.
\end{equation*}
Thus, the fact that $|x-z_e|=d_\Ga(x)+r_e$ gives the first inequality in \eqref{eq:U-V-barriers}, by recalling the definiton of $U^\ve$.
\par 
To obtain the second inequality in \eqref{eq:U-V-barriers} we proceed differently whether $x\in\Om_{r_i}$ or not. Indeed, if $x\in\Om_{r_i}$, there exists $z_i\in\Om$ such that \eqref{uniform-balls-condition} holds for some $y\in\Ga$ and $x\in B_{r_i}(z_i)$; moreover, since $\pa B_{d_\Ga(x)}(x)\cap \pa B_{r_i}(z_i)=\{y\}$, we observe that $x$ lies in the segment joining $y$ to $z_i$, and hence $|x-z_i|=r_i-d_\Ga(x)$. Again, by using the comparison principle and the expression in \eqref{ball solution formula}, we get that 
\begin{equation*}
u^\ve(x)\le
\frac{\int_{0}^{\pi}e^{\sqrt{p'}\,\cos\te\,\frac{|x-z_i|}{\ve}}(\sin\te)^\al\,d\te}{\int_{0}^{\pi}e^{\sqrt{p'}\,\cos\te\,\frac{r_i}{\ve}}(\sin\te)^\al\,d\te},
\end{equation*}
that, by using the definition of $V^\ve$ and the fact that $|x-z_i|=r_i-d_\Ga(x)$, leads to the second inequality in \eqref{eq:U-V-barriers}.
\par
If $x\in\Om\setminus\ol\Om_{r_i}$, we just note that the expression of $V^\ve$ was already obtained in Lemma \ref{elliptic-control-from-above}. 
\par
The case $p=\infty$ can be treated with similar arguments. 
\end{proof}

\subsection{Asymptotics for $q$-means}

From now on, in order to use the function $\Pi_\Ga$ defined in \eqref{eq:Pi-gamma},  we assume that $\Om$ is a domain of class $C^2$ (not necessarily bounded). 
\par
First, we recall from \cite[Lemma 2.1]{MS-PRSE} the following geometrical lemma.

\begin{lem}
\label{lem:geometrical-asymptotics}
Let $x\in\Om$ and assume that, for $R>0$, there exists $y_x\in\Ga$ such that $\ol{B_R(x)}\cap(\RR^N\setminus\Om)=\{y_x\}$ and that $\ka_j(y_x)<1/R$ for $j=1,\dots,N-1$.
\par 
Then, it holds that
$$
\lim_{s\to 0^+}\frac{\cH_{N-1}(\Ga_s\cap B_R(x))}{s^{\frac{N-1}{2}}}=\frac{\om_{N-1}\,(2R)^{\frac{N-1}{2}}}{(N-1)}\left[\Pi_\Ga(y_x)\right]^{-\frac1{2}}, 
$$
where $\cH_{N-1}$ denotes $(N-1)$-dimensional Hausdorff measure and $\om_{N-1}$ is the surface area of a unit sphere in $\RR^{N-1}$.
\end{lem}

The next lemma gives the asymptotic formula for $\ve\to 0^+$ for the $q$-mean on $B_R(x)$ of a quite general class of functions, which includes both the barriers $U^\ve$ and $V^\ve$. 

\begin{lem}
\label{lem:q-mean}
Set $1<q<\infty$,.Let $x\in\Om$ and assume that, for $R>0$, there exists $y_x\in\Ga$ such that $\ol{B_R(x)}\cap(\RR^N\setminus\Om)=\{y_x\}$ and that $\ka_j(y_x)<1/R$ for $j=1,\dots,N-1$.
\par
Let $\{\xi_n\}_{n\in\NN}$ and $\{f_n\}_{n\in\NN}$ be sequences such that
\begin{enumerate}[(i)]
\item
$\xi_n>0$ and $\xi_n\to 0$ as $n\to\infty$;
\item 
$f_n:[0,\infty)\to[0,\infty)$ are decreasing functions;
\item
$f_n$ converges to a function $f$ almost everywhere as $n\to\infty$;
\item
it holds that
\begin{equation}
  \label{integral-condition}
\lim_{n\to\infty}
\int_{0}^{\infty}f_n(\tau)^{q-1}\,\tau^{\frac{N-1}{2}}\,d\tau=\\
\int_{0}^{\infty}f(\tau)^{q-1}\,\tau^{\frac{N-1}{2}}\,d\tau,
\end{equation}
and the last integral converges.
\end{enumerate}
For some $1<q<\infty$, let $\mu_{q,n}(x)$ be the $q$-mean of $f_n(d_\Ga/\xi_n)$ on $B_R(x)$.
\par
Then we have:
\begin{multline}
\label{eq:q-mean-barrier}
\lim_{n\to\infty}
  \left(\frac{R}{\xi_n}\right)^{\frac{N+1}{2(q-1)}}\mu_{q,n}(x)=\\
\left\{\frac{2^{-\frac{N+1}{2}}N!}{\Ga\left(\frac{N+1}{2}\right)^2}\int_{0}^{\infty}f(\tau)^{q-1}\tau^{\frac{N-1}{2}}\,d\tau\right\}^{\frac1{q-1}}\left[\Pi_\Ga(y_x)\right]^{-\frac1{2(q-1)}}.
\end{multline}
\end{lem}

\begin{proof}
  From \cite{IMW}, we know that $\mu_n=\mu_{q,n}(x)$ is the only root of the equation
  \begin{equation}
\label{characterization-mu}
    \int_{B_R(x)}\left[f_n(d_\Ga/\xi_n)-\mu_n\right]_+^{q-1}\,dy=
\int_{B_R(x)}\left[\mu_n-f_n(d_\Ga/\xi_n)\right]_+^{q-1}\,dy,
  \end{equation}
where we mean $[t]_+=\max(0, t)$.
\par
Thus, if we set
$$
\Ga_\si=\{y\in B_R:\, d_\Ga(y)=\si\},
$$
by the co-area formula 
we get that
\begin{equation*}
  \int_{B_R(x)}\left[f_n(d_\Ga/\xi_n)-\mu_n\right]_+^{q-1}\,dy=\\
\int_{0}^{2R}\left[f_n(\si/\xi_n)-\mu_n\right]_+^{q-1}\cH_{N-1}\left(\Ga_\si\right)\,d\si,
\end{equation*}
that, after the change of variable $\si=\xi_n \tau$ and easy manipulations, leads to the formula:
\begin{multline*}
  \int_{B_R(x)}\left[f_n(d_\Ga/\xi_n)-\mu_n\right]_+^{q-1}\,dy=\\
\xi_n^{\frac{N+1}{2}}
\int_{0}^{2R/\xi_n}\left[f_n(\tau)-\mu_n\right]_+^{q-1}\tau^{\frac{N-1}{2}}\left[\frac{\cH_{N-1}\left(\Ga_{\xi_n\tau}\right)}{\left(\xi_n\tau\right)^{\frac{N-1}{2}}}\right]\,d\tau.
\end{multline*}
\par 
Therefore, since $\mu_n\to 0$ as $n\to\infty$, an inspection of the integrand at the right-hand side, assumptions (i)-(iv), and Lemma \ref{lem:geometrical-asymptotics} make it clear that we can apply the generalized dominated convergence theorem (see \cite{LL}) to infer that
\begin{multline}
\label{eq:formula-A-plus}
\lim_{n\to\infty}
\xi_n^{-\frac{N+1}{2}}\int_{B_R(x)}\left[f_n(d_\Ga/\xi_n)-\mu_n\right]_+^{q-1}\,dy=\\
\frac{(2R)^{\frac{N-1}{2}}\om_{N-1}}{(N-1)\sqrt{\Pi_\Ga(y_x)}}\int_{0}^{\infty}f(\tau)^{q-1}\tau^{\frac{N-1}{2}}\,d\tau.
\end{multline}
\par 
Next, by employing again the co-area formula, the right-hand side of \eqref{characterization-mu} can be re-arranged as
\begin{multline*}
\label{eq:A-minus-inequality}
  \int_{B_R(x)} [\mu_n-f_n(d_\Ga/\xi_n)]_+^{q-1}\,dy=
\mu_n^{q-1}\int_{0}^{2R} \left[1-\frac{f_n(\si/\xi_n)}{\mu_n}\right]_+^{q-1}\cH_{N-1}\left(\Ga_\si\right)\,d\si,
\end{multline*}
that leads to the formula
\begin{equation}
  \label{eq:formula-A-minus}
  \lim_{\ve\to 0^+}
\mu_\ve^{1-q}\int_{B_R(x)} [\mu_n-f_n(d_\Ga/\xi_n)]_+^{q-1}\,dy=
|B_R|,
\end{equation}
by dominated convergence theorem, if we can prove that
\begin{equation}
\label{pointwise}
\frac{f_n(\si/\xi_n)}{\mu_n}\to 0 \ \mbox{ as } \ n\to \infty,
\end{equation}
for almost every $\si\ge 0$.
Then, after straightforward computations, \eqref{eq:q-mean-barrier} will follow 
by putting together \eqref{characterization-mu}, \eqref{eq:formula-A-plus} and \eqref{eq:formula-A-minus}.
\par
We now complete the proof by proving that \eqref{pointwise} holds. From \eqref{characterization-mu}, \eqref{eq:formula-A-plus}, and the fact that
$$
\int_{B_R(x)}\left[\mu_n-f_n(d_\Ga/\xi_n)\right]_+^{q-1}\,dy\le \mu_n^{q-1} |B_R|,
$$
we have that there is a positive constant $c$ such that
$$
\mu_n^{1-q}\le c\,\xi_n^{-\frac{N+1}{2}}.
$$
Also, for every $\si>0$ we have that
\begin{multline*}
\int_{\si/2\xi_n}^\infty f_n(\tau)^{q-1}\,\tau^\frac{N-1}{2} d\tau\ge
\int_{\si/2\xi_n}^{\si/\xi_n} f_n(\tau)^{q-1}\,\tau^\frac{N-1}{2} d\tau\ge \\
\frac{2(1-2^{-\frac{N+1}{2}})}{N+1}\,f_n(\si/\xi_n)^{q-1}\left(\frac{\si}{\xi_n}\right)^\frac{N+1}{2}\ge \\
\frac{2(1-2^{-\frac{N+1}{2}})}{c\,(N+1)}\,\si^\frac{N+1}{2}\left\{\frac{f_n(\si/\xi_n)}{\mu_n}\right\}^{q-1}.
\end{multline*}
Thus, \eqref{pointwise} follows, since the first term of this chain of inequalities converges to zero as $n\to\infty$, under our assumptions on $f_n$ and $\xi_n$, in virtue of the generalized dominated convergence theorem.
\end{proof}

\begin{rem}
{\rm
The case $q=\infty$ is simpler. From \cite{IMW} and then the monotonicity of $f_n$ we obtain that:
\begin{multline*}
\mu_{\infty,n}(x)=\frac1{2}\left\{\min_{B_R(x)}f_n\left(d_\Ga/\xi_n\right)+\max_{B_R(x)}f_n\left(d_\Ga/\xi_n\right)\right\}= \\
\frac1{2}\left\{f_n\left(2R/\xi_n\right)+f_n(0)\right\}.
\end{multline*}
Thus, if we replace the assumptions (iii) and (iv) by $f_n(0)\to f(0)$ as $n\to\infty$, we conclude that
$\mu_{\infty,n}(x)\to f(0)/2$, since $f_n\left(2R/\xi_n\right)\to 0$ as $n\to\infty$. 
}
\end{rem}

\begin{thm}
  \label{th:q-mean}
Set $1<p\le \infty$. Let $x\in\Om$ be such that $B_R(x)\subset\Om$ and $\ol{B_R(x)}\cap(\RR^N\setminus\Om)=\{y_x\}$; suppose that $k_j(y_x)<\frac1{R}$, for every $j=1,\dots,N-1$.
\par 
Let $u^\ve$ be the bounded (viscosity) solution of \eqref{G-elliptic}-\eqref{elliptic-boundary} and, for $1<q\le\infty$, let $\mu_{q,\ve}(x)$ be the $q$-mean of $u^\ve$ on $B_R(x)$.
\par 
Then, if $1<q<\infty,$ we have that \eqref{q-mean} holds, that is
\begin{equation*}
  \label{eq:q-mean}
\lim_{\ve\to 0^+}
  \left(\frac{\ve}{R}\right)^{-\frac{N+1}{2(q-1)}}\mu_{q,\ve}(x)=
\frac{c_{N,q}}{\left\{(p')^{\frac{N+1}{2}}
\Pi_\Ga(y_x)\right\}^{\frac{1}{2(q-1)}}},
\end{equation*}
where
$$
c_{N,q}=\left\{\frac{2^{-\frac{N+1}{2}}N!}{(q-1)^{\frac{N+1}{2}}\Ga\left(\frac{N+1}{2}\right)}\right\}^{\frac1{q-1}}.
$$
\par 
If $q=\infty$, we simply have that $\mu_{\infty,\ve}(x)\to 1/2$ as $\ve\to 0$.
\end{thm}

\begin{proof}
We have that $\mu_{q,\ve}^{U^\ve}(x)\le \mu_{q,\ve}(x)\le \mu_{q,\ve}^{V^\ve}(x)$ by the monotonicity properties
of the $q$-means, where with $\mu_{q,\ve}^{U_\ve}$ and $\mu_{q,\ve}^{V^\ve}$ we denote the $q$-mean of $U^\ve(d/\ve)$ and $V^\ve(d/\ve)$ on $B_R(x)$.
Hence, in order to prove \eqref{q-mean}, we only need to apply Lemma \ref{lem:q-mean} to  $f_n=U^{\ve_n}$ and $f_n=V^{\ve'_n}$, where the vanishing sequences $\ve_n$ and $\ve'_n$ are chosen so that the $\liminf$ and $\limsup$ of $\left(\ve/R\right)^{-\frac{N+1}{2(q-1)}}\mu_{q,\ve}(x)$
as $\ve\to 0$ are attained along them, respectively.
\par
By an inspection, it is not difficult to check that $f_n=U^{\ve_n}$ and $f_n=V^{\ve'_n}$, with $\xi_\ve=\ve/\sqrt{p'}$ and $f(\tau)=e^{-\tau}$, satisfy the relevant assumptions of Lemma \ref{lem:q-mean}, by applying, in particular, Lemma \ref{lem:mollifier} for (iii) and the dominated convergence theorem for (iv).
\end{proof} 

\appendix

\section{Technical lemmas}
\label{sec:appendix-technical-lemma}

Here, we collect two useful lemmas.

\begin{lem}[One-dimensional asymptotics]
\label{lem:technical}
For $\al>-1$ and $\si>0$, let
$$
f(\si)=\int_0^\infty e^{-\si (\cosh\te-1)} (\sinh\te)^\al d\te.
$$ 
\par
Then, $f$ is continuous in $(0,\infty)$ and
\begin{equation*}
\label{eq:asymptotics-to-infinity}
f(\si)=
2^\frac{\al-1}{2} \Ga\left(\frac{\al+1}{2}\right)\si^{-\frac{\al+1}{2}} \bigl\{ 1+O(1/\si)\bigr\} \ \mbox{ as } \ \si\to\infty.
\end{equation*}
Moreover, if $\si\to 0$, we have that
$$
f(\si)=\begin{cases}
\si^{-\al}\, \Ga\left(\al\right)\bigl\{ 1+o(1)\bigr\}  \ &\mbox{ if } \ \al>0, 
\vspace{6pt}
\\
\log(1/\si)+O(1) \ &\mbox{ if } \ \al=0, 
\vspace{2pt}
\\
\frac{\sqrt{\pi}}{2\,\sin(\al\pi/2)}\,\frac{\Ga\left(\frac{\al+1}{2}\right)}{\Ga\left(\frac{\al}{2}+1\right)}+o(1)  \ &\mbox{ if } \ -1<\al<0.
\end{cases}
$$
\end{lem}

\begin{proof}
By the change of variable $\tau=\si\, (\cosh\te-1)$ we get:
$$
f(\si)=\frac1{\si}\int_0^\infty e^{-\tau} \left(\frac{2\tau}{\si}+\frac{\tau^2}{\si^2}\right)^\frac{\al-1}{2} \!\!d\tau.
$$
When $\si\to\infty$, our claim follows by writing
$$
f(\si)=2^\frac{\al-1}{2} \si^{-\frac{\al+1}{2}} \int_0^\infty e^{-\tau}\left(\tau+\frac{\tau^2}{2\si}\right)^\frac{\al-1}{2}\!\!d\tau.
$$
When $\si\to 0$ and $\al>0$, our claim follows by writing
$$
f(\si)=\si^{-\al} \int_0^\infty e^{-\tau}\left(\tau^2+2\si\,\tau\right)^\frac{\al-1}{2}\!\!d\tau,
$$
since in this case the (limiting) integral converges near zero.
\par
For $-1<\al\le 0$, we use the formula
$$
f(\si)=\frac1{\sqrt{\pi}}\,\Ga\left(\frac{\al+1}{2}\right)\left(\frac{\si}{2}\right)^{-\frac{\al}{2}} e^\si\, K_{\al/2}(\si),
$$
where $K_{\al/2}(\si)$ is the modified Bessel's function of the second kind of order $\al/2$ (see \cite[Formula 9.6.23]{AS}). Then, \cite[Formula 9.6.9]{AS} and \cite[Formula 9.6.13]{AS} give our claims for $\al=0$ and $-1<\al<0$, respectively.
\end{proof}

\begin{lem}[Mollifier]
\label{lem:mollifier}
 Let $g:\RR\to \RR$ be a bounded and continuous function.
\par
Then, we have that
  \begin{equation*}
\label{nu-t}
\lim_{\si\to\infty}
\int_{0}^{\infty}g(\te)\left[\frac{e^{-\si(\cosh\te-1)}(\sinh\te)^\al}{\int_{0}^{\infty}e^{-\si(\cosh\te-1)}(\sinh\te)^\al}\right]\,d\te=
g(0)
  \end{equation*}
and
\begin{equation*}
\label{mu-t}
  \lim_{\si\to \infty}
\int_{0}^{\pi}g(\te)\left[\frac{e^{-\si(1-\cos\te)}(\sin\te)^\al}{\int_{0}^{\pi}e^{-\si(1-\cos\te)}(\sin\te)^\al\,d\te}\right]\,d\te=
g(0).
\end{equation*}
\end{lem}

\begin{proof}
The two formulas in the statement follow by observing that in both cases the relevant integral can be written as 
$$
\int_0^\infty j_\si(\te)\, g(\te)\,d\te,
$$
where $j_\si$ is a {\it mollifier}, that has the salient properties:
\begin{multline*}
j_\te\ge 0, \quad \int_0^\infty j_\si(\te)\,d\te=1 \ \mbox{ for any } \ \si>0 \ \mbox{ and } \\ \lim_{\si\to\infty}\int_\de^\infty j_\si(\te)\,d\te=0 \ \mbox{ for any } \ \de>0.
\end{multline*}
The last property easily follows from the substitution $\tau=\si\,(\cosh\te-1)$ or $\tau=\si\,(1-\cos\te)$.
\end{proof}

\section*{Aknowledgements}
The paper was partially supported by the Gruppo Nazionale Analisi Matematica Probabilit\`a e Applicazioni (GNAMPA) of the Istituto Nazionale di Alta Matematica (INdAM).

\end{document}